\documentclass[]{article}
\usepackage{indentfirst}
\usepackage{amsmath}
\usepackage{authblk}
\usepackage{geometry}
\usepackage[utf8]{inputenc}
\usepackage[francais]{babel}
\usepackage{ucs}
\usepackage{statex}

\newtheorem{theorem}{Theorem}
\newtheorem{proposition}{Proposition}
\newtheorem{lemma}{Lemma}
		
\newenvironment{proof}[1][Proof]{\begin{trivlist}
		\item[\hskip \labelsep {\bfseries #1}]}{\end{trivlist}}
\newenvironment{definition}[1][Definition]{\begin{trivlist}
		\item[\hskip \labelsep {\bfseries #1}]}{\end{trivlist}}

\newenvironment{remark}[1][Remark]{\begin{trivlist}
		\item[\hskip \labelsep {\bfseries #1}]}{\end{trivlist}}

\title{Generalized Dirichlet series of n variables associated with automatic sequences}
\author{Shuo LI\\

\normalsize{shuo.li@imj-prg.fr}\\}
\date {}

\begin{document}	
\maketitle

\section{Introduction}

The propose of this article is to give a sufficient condition for the meromorphic continuation of Dirichlet series of form $\sum_{\ul{x}\in \mathbf{N}^n_+} \frac{a_{\ul{x}}\prod_{i=1}^nx_i^{\mu_i}}{P(\ul{x})^s}$, where $(a_{\ul{x}})_{\ul{x}\in \mathbf{N}^n_+}$ is a $q$-automatic sequence of $n$ parameters, $\mu_i \in \mathbf{Z}_+$ and $P: \mathbf{N}^n \to \mathbf{R}$ a polynomial, such that $P$ does not have zeros in $\mathbf{Q}^{n}_{+}$. Some specific cases for $n=1$ are studied in this article as examples to show the possibility to have an holomorphic continuity on the whole complex plane. Some equivalences between infinite products are also built as consequences of these results.\\

Remarking that constant sequences are a kind of particular automatic sequences,  the Dirichlet sequences in the form $\sum_{(x_1, x_2...x_n) \in \mathbf{N}^I_{+}} \frac{x_1^{\mu_1}x_2^{\mu_2}...x_n^{\mu_n}}{p(x_1,x_2,..x_n)^s}$ have been widely studied, where $(\mu_1, \mu_2,... \mu_n) \in \mathbf{N}^n_{+}$ and $p$ is a $n$-variable function. R.H. Mellin \cite{Mellin} firstly proved in 1900 that the functions above have a meromorphic continuation to the whole complex plane when $\mu_i=0$ for all indexes $i$, then K. Mahler \cite{Mah} generalized the result to the case that $\mu_i$ are arbitrary positive integers when the polynomial satisfies the elliptic condition in 1927. In 1987, P. Sargos \cite{sargos} proved that the condition ``$\lim |p(x_1, x_2...x_n)| \to \infty$ when $|(x_1, x_2...x_n)| \to \infty$ and $p$ is non-degenerate'' is sufficient for these Dirichlet sequences to 
have a meromorphic continuation. In 1997, D. Essouabri \cite{ess} generalized the condition to the case, $\frac{\partial_{\ul{\mu}}p(\ul{x})}{p(\ul{x})} =\mathcal{O}(1), \forall \ul{\mu} \in \mathbf{Z}_+^n$.The Dirichlet sequences of the form $\sum_{n=1}^{\infty} \frac{a_n}{n^s}$ have been studied in \cite{ALLOUCHE2000}, and our work is a natural generalization of the results in the above article by using the same method of calculations.

\section {Notation, definitions and basic properties of automatic sequences}

Here we define some notation used in this article. We let $\ul{x}$ denote an $n$-tuple ${(x_1, x_2 ... x_n)}$. We say $\ul{x} \geq \ul{y}$ (resp. $\ul{x} > \ul{y}$) if and only if $\ul{x}-\ul{y} \in \mathbf{R}_{+}^n$(resp. $\ul{x}-\ul{y} \in \mathbf{R}_+^n$), and we have an analogue definition for the symbol $\leq$ (resp. $<$). We let $\ul{x}^{\ul{\mu}}$ denote the $n$-tuple ${(x_1^{\mu_1}, x_2^{\mu_2} ... x_n^{\mu_n})}$. For a constant $c$, we let $\ul{c}$ denote the tuple $(c,c...c)$ and for two tuples $\ul{x}$ and $\ul{y}$, we let $<\ul{x},\ul{y}>$ denote the real number $\sum_{i=1}^{n}x_iy_i$. For an $n$-tuple ${(x_1, x_2 ... x_n)}$, we let $||\; ||_d$ denote the norm $d$ and let $||\;||$ or $||\;||_2$ denote the norm 2.

\begin{definition}
Let $q \geq 2$ be an integer. A sequence $(a_{\ul{x}})_{\ul{x} \geq \ul{0}}$ with values
in the set $\mathcal{A}$ is called $q$-automatic if and only if its $q$-kernel $\mathcal{N}_q ((a_{\ul{x}})_{\ul{x} \geq \ul{0}})$ is finite,
where the $q$-kernel of the sequence $(a_{\ul{x}})_{\ul{x} \geq \ul{0}}$ is the set of subsequences
defined by
$$\mathcal{N}_q ((a_{\ul{x}})_{\ul{x} \geq \ul{0}})=\left\{(m_1, m_2 ... m_n)\longmapsto a_{(q^km_1+l_1,q^km_2+l_2, ..., q^km_n+l_n)}; k\geq 0, \ul{(0)}\leq \ul{l} \leq \ul{(q^k-1)}\right\}.$$
\end{definition} 

\begin{remark}
	A $q$-automatic sequence necessarily takes finitely many
	values. Hence
	we can assume that the set $\mathcal{A}$ is finite.
\end{remark}

Because of the definition of $q$-automatic with $n$ variables, there are some basic properties.

\begin{theorem}
	Let $q\geq 2$ be an integer and  $(a_{\ul{x}})_{\ul{x}\geq \ul{0}}$ be a sequence with values
	in $\mathcal{A}$. Then, the following properties are equivalent:
	
	(i) The sequence  $(a_{\ul{x}})_{\ul{x}\geq \ul{0}}$ is $q$-automatic
	
	(ii) There exists an integer $t \geq 1$ and a set of $t$ sequences $\mathcal{N}^{'}=\left\{(a_{\ul{x}}^1)_{\ul{x}\geq \ul{0}},...,(a_{\ul{x}}^t)_{\ul{x}\geq \ul{0}}\right\})$ such that
	
	- the sequence $(a_{\ul{x}}^1)_{\ul{x}\geq \ul{0}}$ is equal to the sequence $(a_{\ul{x}})_{\ul{x}\geq \ul{0}}$
	
	- the set $\mathcal{N}^{'}$ is closed under the maps $ (a_{\ul{x}})_{\ul{x}\geq \ul{0}} \longmapsto (a_{q\ul{x}+\ul{y}})_{\ul{x}\geq \ul{0}}$ for $\ul{0} \leq \ul{y} \leq \ul{q-1}$
	
	(iii) There exist an integer $t\geq 1$ and a sequence $(A_{\ul{x}})_{\ul{x}\geq \ul{0}}$ with values in $\mathcal{A}^t$, that we denote as a column vector, as $(A_{1,1...1}, A_{2,1...1}, A_{1,2...1}...A_{1,1...2}, A_{2,2...1}...)^t$. There exist $q^n$ matrices of size $t \times t$, say $M_{1,1...1}, M_{1,2...1}...M_{q,q..q}$ , with the property that each row of each $M_i$ has exactly one entry equal to $1$, and the other $t-1$ entries equal to $0$, such that:
	
	- the first component of the vector  $(A_{\ul{x}})_{\ul{x}\geq \ul{0}}$ is the sequence  $(a_{\ul{x}})_{\ul{x}\geq \ul{0}}$
	
	- for each $\ul{y}$ such that $\ul{0} \leq \ul{y} \leq \ul{q-1}$, the equality $A_{q\ul{x}+\ul{y}}=M_{\ul{y}}A_{\ul{x}}$ holds.
\end{theorem}
	
\begin{proof}
	It is a natural consequence of the finiteness of the set $\mathcal{N}^{'}$, see for example  \cite{salon}.
\end{proof}

\begin{proposition}
	Let $(a_{\ul{x}})_{\ul{x}\geq \ul{0}}$ be a $q$-automatic sequence and $(b_{\ul{x}})_{\ul{x}\geq \ul{0}}$ be a periodic sequence of period $\ul{c}$. Then the sequence $(a_{\ul{x}} \times b_{\ul{x}})_{\ul{x}\geq \ul{0}}$ is also $q$-automatic and its $q$-kernel can be completed in such a way that all transition matrices of the maps $(a_{\ul{x}} \times b_{\ul{x}})_{\ul{x}\geq \ul{0}} \longmapsto (a_{q\ul{x}+\ul{y}} \times b_{q\ul{x}+\ul{y}})_{\ul{x}\geq \ul{0}}$ on the new set are independent on the choice of the values taken by the sequence $(b_{\ul{x}})_{\ul{x}\geq \ul{0}}$.
\end{proposition}

\begin{proof}
	As $(a_{\ul{x}})_{\ul{x}\geq \ul{0}}$ is a $q$-automatic sequence, we let\\$\mathcal{N}_a : \left\{(a_{\ul{x}}^{(1)})_{\ul{x}\geq \ul{0}},(a_{\ul{x}}^{(2)})_{\ul{x}\geq \ul{0}}, ..., (a_{\ul{x}}^{(l)})_{\ul{x}\geq \ul{0}}\right\}$ denote its $q$-kernel. The sequence $(b_{\ul{x}})_{\ul{x}\geq \ul{0}}$ is a periodic sequence, thus it is also an $q$-automatic sequence, we let $\mathcal{N}_b : \left\{(b_{\ul{x}}^{(1)})_{\ul{x}\geq \ul{0}},(b_{\ul{x}}^{(2)})_{\ul{x}\geq \ul{0}}, ..., (b_{\ul{x}}^{(s)})_{\ul{x}\geq \ul{0}}\right\}$ denote the $q$-kernel of $(b_{\ul{x}})_{\ul{x}\geq \ul{0}}$.  As both of the $q$-kernel are finite, we can conclude that the set of Cartesian product of these two above sets is finite: $$\mathcal{N}_{ab} : \left\{(a_{\ul{x}}^{(i)} \times b_{\ul{x}}^{(j)})_{\ul{x}\geq \ul{0}}| 0 \leq i \leq l, 0 \leq j \leq s \right\}$$ which is the $q$-kernel of the sequence $(a_{\ul{x}} \times b_{\ul{x}})_{\ul{x}\geq \ul{0}}$. \\

To the completion, we remark that there is a onto map from $\mathcal{N}_b'$ to $\mathcal{N}_b$, where $\mathcal{N}_b'$ is the $q$-kernel of the periodic sequence $(I_{\ul{x}})_{\ul{x}\geq \ul{0}}$ defined by $$I_{(m_1,m_2,...,m_n)}=(y_1,y_2,...,y_n) \ \text{where}\ y_i\equiv m_i \mod c, 1\leq i \leq n $$
	and the map is defined as $$\mathcal{N}_b' \to \mathcal{N}_b : (I_{\ul{x}})_{\ul{x} \geq \ul{0}} \to (b_{I_{\ul{x}}})_{\ul{x} \geq \ul{0}}.$$ 
So it is enough to work on the finite set $\mathcal{N}_a \times \mathcal{N}'_b$.
\end{proof}

Let us consider the Dirichlet series $f(s)=\sum_{\ul{x}\in \mathbf{N}^n/\ul{0}} \frac{a_{\ul{x}}}{p(\ul{x})^s}$, Where $a_{\ul{x}}$ is $q$-automatic, a necessary condition of the convergence of such series is that $|p(\ul{x})| \to \infty$ when $||\ul{x}|| \to \infty$, here we want to find a sufficient condition.

An achievable assumption for $f(s)$ to be meromorphic is that the polynomial $p$ is elliptic, which means that, if the degree of $p$ is $d$ then the homogeneous polynomial $p_d(\ul{x})$ of $p(\ul{x})$ satisfies the condition 

$$p_d(\ul{x})>0, \; \forall \ul{x} \in [0, \infty[^n \backslash \left\{(0,0,...0)\right\}.$$

Before announcing the main theorem, we would like to study some properties of elliptic polynomials:

\begin{lemma}
Let $(r_1,r_2,..,r_n)$ be a vector on $\mathbf{Z}^n_{+}$ such that $\sum_{i=1}^nr_i <d$ and $\ul{x} \in \mathbf{R}_+^n$ then:

$$\frac{x_1^{r_1}x_2^{r_2}...x_n^{r_n}}{\sum_{i=1}^nx_i^d} =\mathcal{O} (||x||^{\sum_{i=1}^{n}r_i-d}), \; \text{when} \; ||x|| \to \infty.$$

\end{lemma} 

\begin{proof}
It is enough to see the following inequality:

$$\sum_{i=1}^nx_i^d=\sum_{i=1}^nr_i\frac{1}{r_i}x_i^d =\sum_{i=1}^n\sum_{j=1}^{r_i}\frac{1}{r_i}x_i^d \geq (\prod_{i=1}^{n}\frac{1}{r_i^{r_i}}x_i^{dr_i})^{\frac{1}{\sum_{i=1}^nr_i}},$$
which is from the inequality of arithmetic and geometric means. With the assumption $\sum_{i=1}^nr_i <d$ and the equivalences between norms, we conclude $\frac{x_1^{r_1}x_2^{r_2}...x_n^{r_n}}{\sum_{i=1}^n x_i^d} \leq \prod_{i=1}^{n}r_i^{r_i}||x||_d^{\sum_{i=1}^{n}r_i-d}=\mathcal{O} (||x||^{\sum_{i=1}^{n}r_i-d}), \; \text{when} \; ||x|| \to \infty$.
\end{proof}

\begin{lemma}
If a polynomial $p$ is elliptic of degree $d$ and $p_d$ is the homogeneous polynomial of degree $d$ of $p$, then:

 (i) all coefficients of terms $x_1^d,x_2^d,...,x_n^d$ are positive;

(ii) $p_d(\ul{x}) \to\infty \; \text{when} \; ||x||\to \infty$;

(iii) there exists a positive number $\alpha$ such that $\forall \ul{x} \in [0, \infty[^n \backslash \left\{(0,0,...0)\right\}$, $p_d'(\ul{x})=p_d(\ul{x})-\alpha\sum_{i=1}^{n}x_i^d>0$;

(iv) $p(\ul{x}) \to\infty \; \text{when} \; ||x||\to \infty$;

(v) there exists a positive number $\alpha$ such that $p'(\ul{x})=p(\ul{x})-\alpha\sum_{i=1}^{n}x_i^d \to \infty$ when $||x|| \to \infty$.

\end{lemma} 

\begin{proof}

Assertion (i) is straightforward by evaluating the function at\\
 $(1,0,...,0),(0,1,...,0),...,(0,0,...,1)$.\\
 For (ii),  let us consider the set 
$$\left\{p_d(\ul{x})|\forall \ul{x} \in [0, \infty[^n \backslash \left\{(0,0,...0)\right\}, ||\ul{x}||_d=1\right\},$$this set is closed because of closed map lemma and every element inside is larger then $0$, so that such a set admits a non-zero infimum, let us denote it by $\epsilon$. Then for an arbitrary $\ul{x}$, we have $$p_d(\ul{x})=||\ul{x}||_d^dp_d(\frac{\ul{x}}{||\ul{x}||_d})\geq ||\ul{x}||_d^d \epsilon.$$ 
For (iii), Setting $$p_d'(\ul{x})=p_d(\ul{x})-\frac{\epsilon}{2}\sum_{i=1}^{n}x^d_i,$$ it is easy to check $p_d'(\ul{x})\geq ||x||_d^d\epsilon-\frac{\epsilon}{2}\sum_{i=1}^{n}x_i^d>0,\forall \ul{x} \in [0, \infty[^n \backslash \left\{(0,0,...0)\right\}$ and $||\ul{x}||_d=1$, then it follows that $$p_d'(\ul{x})=||\ul{x}||_d^d\frac{p_d'(\ul{x})}{||\ul{x}||_d^d}=||\ul{x}||_d^dp_d'(\frac{\ul{x}}{||\ul{x}||_d})>0,$$for all $\ul{x}$ in the set  $[0, \infty[^n \backslash \left\{(0,0,...0)\right\}$.

For (iv) it is enough to point out that each monomial of degree smaller then $d$ can be bounded above by a term of the form $\alpha\sum_{i=1}^{n}x_i^d$ because of Lemma 1, and conclude by (iii).

(v) is a direct consequence of (iii) and (iv).

\end{proof}

\begin{lemma}
Let $P$ be an elliptic polynomial of degree $d$ and let $p$ be a polynomial with a degree smaller than $d$, then there exists un integer $C$ such that for all $n$-tuples $\ul{x}\in \mathbf{N}_+^n$ with $<\ul{x},\ul{1}>>C$,
$$\frac{|p(\ul{x})|}{P(\ul{x})}<\frac{1}{P(\ul{x})^{\frac{1}{2d}}}$$

\end{lemma} 

\begin{proof}

Let us consider the polynomial $q$ defined by $q=P^{2d-1}-p^{2d}$. We can check that $q$ is of degree $2d^2-d$ and its homogeneous polynomial is uniquely defined by the one of $P$, which is from the fact that the degree of $p^{2d}$ is at most $2d^2-2d$. So that $q$ is an elliptic polynomial, thus there exists un integer $C$ such that $P^{2d-1}(\ul{x})-p^{2d}(\ul{x})>0$ if $<\ul{x},\ul{1}>>C$. As a result, for all $\ul{x}$ with $<\ul{x},\ul{1}>>C$, $$(\frac{p(\ul{x})}{P(\ul{x})})^{2d}=\frac{1}{P(\ul{x})}\frac{p(\ul{x})^{2n}}{P(\ul{x})^{2n-1}}<\frac{1}{P(\ul{x})}.$$

\end{proof}

\section {Proof of the meromorphic continuation}

In this section we prove the main result.

\begin{theorem}
	Let $p$ be an elliptic polynomial of $n$ variables and $(a_{\ul{x}})_{\ul{x} >\ul{0}}$ be $q$-automatic, then for a given $n$-tuple $\ul{\mu}$, the function $\sum_{(\ul{x}) \in \mathbf{N}^n_+} \frac{a_{\ul{x}}\prod_{i=1}^{n}x_i^{\mu_i}}{p(\ul{x})^{s}}$ admits an abscissa of convergence $\sigma$ such that it converges absolutely on the half plane $\Re(s) > \sigma $ and has a meromorphic continuation on whole complex plane. furthermore, the poles of this function (if any) are located on a finite number of left semi-lattices.
\end{theorem}

This result will be obtained by proving several lemmas successively:

\begin{proposition}
		Let $a_{\ul{x}}$ be a $q$-automatic sequence, and $p(\ul{x})=\sum_{\ul{\alpha}}m_{\ul{\alpha}}x^{\ul{\alpha}}$ be a $n$-variable homogeneous elliptic polynomial of degree $d$,  let $\ul{\mu} \in \mathbf{N}^n_{+}$ be a multi-index, for any $\ul{\beta}$ such that $\ul{0} \leq \ul{\beta} \leq \ul{q}$, define $p_{\ul{\beta}}(\ul{x})=q^{-n}(p(q\ul{x}+\ul{\beta})-p(q\ul{x}))$, then for any $k \in \mathbf{N}$, the function $f_{k,\ul{\beta},\ul{\mu}}: s \longrightarrow \sum_{(\ul{x}) \in \mathbf{N}^n_+} \frac{a_{\ul{x}}p_{\ul{\beta}}(\ul{x})^k\prod_{i=1}^{n}x_i^{\mu_i}}{p(\ul{x})^{s+k}}$ admits an abscissa of convergence $\sigma_{k,\ul{\beta}, \ul{\mu}}$ such that $f_{k,\ul{\beta},\ul{\mu}}$ converges absolutely to an holomorphic function on the right half-plane $\Re(s)>\sigma_{k,\ul{\beta},\ul{\mu}}$.
\end{proposition}

\begin{proof}
	
	We firstly prove that $f_{0,\ul{0},\ul{0}}(s)$ converge when $\Re(s)>n$.
	\begingroup\small\begin{equation}
		\begin{aligned}
			|f_{0,\ul{0},\ul{0}}(s)|=&|\sum_{(\ul{x}) \in \mathbf{N}^n_+} \frac{a_{\ul{x}}}{p(\ul{x})^{s}}|\leq \sum_{(\ul{x}) \in \mathbf{N}^n_+} \frac{|a_{\ul{x}}|}{p(\ul{x})^{\Re(s)}} \leq \sum_{(\ul{x}) \in \mathbf{N}^n_+} \frac{|a_{\ul{x}} |}{(\alpha<\ul{x},\ul{1}>)^{\Re(s)}}\;(*)\\
			&\leq \frac{\max(|a_{\ul{x}}|)}{(\alpha)^{\Re(s)}}(\sum_{<\ul{x},\ul{1}> < n} \frac{1}{<\ul{x},\ul{1}>^{\Re(s)}}+\sum_{<\ul{x},\ul{1}> \geq n} \frac{1}{<\ul{x},\ul{1}>^{\Re(s)}})\\
			&\leq \frac{\max(|a_{\ul{x}}|)}{(\alpha)^{\Re(s)}}(\sum_{<\ul{x},\ul{1}> < n} \frac{1}{<\ul{x},\ul{1}>^{\Re(s)}}+ \sum_{m\geq n}\frac{\binom{m+n-1}{n-1}}{m^{\Re(s)}}) \\
			&\leq \frac{\max(|a_{\ul{x}}|)}{(\alpha)^{\Re(s)}}(\sum_{<\ul{x},\ul{1}> < n} \frac{1}{<\ul{x},\ul{1}>^{\Re(s)}}+ \sum_{m\geq n}\frac{m^{n-1}}{m^{\Re(s)}}) \\
			&\leq
			\frac{\max(|a_{\ul{x}}|)}{(\alpha)^{\Re(s)}}(\sum_{<\ul{x},\ul{1}> < n} \frac{1}{<\ul{x},\ul{1}>^{\Re(s)}}+ \sum_{m\geq n}\frac{1}{m^{\Re(s)+1-n}}).\\
		\end{aligned}
	\end{equation}\endgroup
Inequality $(*)$ is obtained by Lemma 2.2 (iii) and taking $\alpha$ as it was in the lemma; and the sum $\sum_{m\geq n}\frac{1}{m^{\Re(s)+1-n}}$ exists and is bounded when $\Re(s)>n$.

For any $ \ul{\beta}$ such that $\ul{0} \leq \ul{\beta} \leq \ul{q}$,  we remark that $\prod_{i=0}^{n}(x_i+\beta_i)^{k_i}=\sum_{\ul{l} \leq \ul{k}}C_{\ul{l}}\prod_{i=0}^{n}(x_i)^{l_i}$ with $C_{\ul{l}} < q^n$, which shows that all monomials of the polynomial $p_{\ul{\beta}}$ have a degree not larger than $d-1$. Lemma 2.3 leads $|\frac{p_{\ul{\beta}}(\ul{x})}{p(\ul{x})}| < \frac{1}{p(\ul{x})}^{\frac{1}{2d}}$ for all $\ul{x}$ satisfying $<\ul{x},\ul{1}>>C_1$, with $C_1$ defined in Lemma 2.3. While Lemma 2.1 and Lemma 2.2 (v) yield that there exists an integer $l$ such that $|\frac{\prod_{i=1}^{n}x_i^{\mu_i}}{p^l(\ul{x})}| \to 0$ when $|\ul{x}| \to \infty$. As a result, there exists $C_2$ such that for all $<\ul{x},\ul{1}>\geq C_2, |\frac{\prod_{i=1}^{n}x_i^{\mu_i}}{p^l(\ul{x})}| < 1$. Taking $C=\max(C_1,C_2)$,
	\begingroup\small\begin{equation}
		\begin{aligned}
\sum_{(\ul{x}) \in \mathbf{N}^n_+} |\frac{a_{\ul{x}}p_{\ul{\beta}}^k(\ul{x})\prod_{i=1}^{n}x_i^{\mu_i} }{p(\ul{x})^{s+k}}|&\leq \sum_{(\ul{x}) \in \mathbf{N}^n_+} \frac{a_{\ul{x}}\prod_{i=1}^{n}x_i^{\mu_i}}{|p^{\Re(s)}(\ul{x})|}|\frac{p_{\ul{\beta}}(\ul{x})}{p(\ul{x})}|^k \\
&\leq \sum_{<\ul{x},\ul{1}> < C} \frac{a_{\ul{x}}\prod_{i=1}^{n}x_i^{\mu_i}}{|p^{\Re(s)}(\ul{x})|}|\frac{p_{\ul{\beta}}(\ul{x})}{p(\ul{x})}|^k+\sum_{<\ul{x},\ul{1}> \geq C} \frac{a_{\ul{x}}}{p^{\Re(s)-l+\frac{k}{2d}}(\ul{x})}.
		\end{aligned}
	\end{equation}\endgroup
With $k$ a constant in $\mathbf{N}^{+}$, the above function converges to a holomorphic function on the half plane $\Re(s)>n+l-\frac{k}{2d}$. Furthermore, for all $b>k$, $\sum_{<\ul{x},\ul{1}> \geq x_1} |\frac{a_{\ul{x}}p_{\ul{\beta}}^b(\ul{x})}{p^{s-l+b}(\ul{x})}|$ is bounded on this half plane.
\end{proof}
	
\begin{proposition}
	With the same notation as above, if $p$ is an homogeneous polynomial, then the function $F: s \longrightarrow \sum_{(\ul{x}) \in \mathbf{N}^n_+} \frac{a_{\ul{x}}\prod_{i=1}^{n}x_i^{\mu_i}}{p(\ul{x})^{s}}$ admits a meromorphic continuation on the whole complex plane.
\end{proposition}
	
\begin{proof}
In this proof, we consider the $q$-automatic sequence $(a_{\ul{x}})_{\ul{x}\geq \ul{0}}$ as itself multiplied by a constant sequence $(b_{\ul{x}})_{\ul{x}\geq \ul{0}}=1$, which is a $q$-periodic sequence. Because of Proposition 2.1, the $q$-kernel of this sequence admits a completion, we can define a sequence of vectors $(A_{\ul{x}})_{\ul{x}\geq \ul{0}}$ and the matrices of transition on this completion as  in Theorem 2.1.

 For any $\ul{\mu} \in \mathbf{N}^n_{+}$, there exists some $l\in \mathbf{Z}$ such that $<\ul{\mu}, \ul{1}> < ld$ and a constant $N_0 \in \mathbf{N}$ such that $C<N_0nq$, where $C$ is defined as in the previous lemma.
 
\begingroup\small\begin{equation}
		\begin{aligned}
F_{\ul{\mu}}(s)&=\sum_{(\ul{x})\in \mathbf{N}^n_+}\frac{A_{\ul{x}}\prod_{i=1}^{n}x_i^{\mu_i}}{p(\ul{x})^{s}}=\sum_{(\ul{x})<(\ul{N_0q})} \frac{A_{\ul{x}}\prod_{i=1}^{n}x_i^{\mu_i}}{p(\ul{x})^{s}}+\sum_{(\ul{y})<(\ul{q})}\sum_{(\ul{z}) \in \mathbf{N}^n/\left\{\ul{t}<\ul{N_0}\right\}} \frac{A_{q\ul{z}+\ul{y}}\prod_{i=1}^{n}(qz_i+y_i)^{\mu_i}}{p^{s}(q\ul{z}+\ul{y})}\\
			&=\sum_{(\ul{x})<(\ul{N_0q})} \frac{A_{\ul{x}}\prod_{i=1}^{n}x_i^{\mu_i}}{p(\ul{x})^{s}}+\sum_{(\ul{y})<(\ul{q})}\sum_{(\ul{z}) \in \mathbf{N}^n/\left\{\ul{t}<\ul{N_0}\right\}} \frac{A_{q\ul{z}+\ul{y}}\prod_{i=1}^{n}(qz_i)^{\mu_i}}{p^{s}(q\ul{z}+\ul{y})}\\
&+\sum_{(\ul{\psi})<(\ul{\mu})}\sum_{(\ul{y})<(\ul{q})}\sum_{(\ul{z}) \in \mathbf{N}^n/\left\{\ul{t}<\ul{N_0}\right\}} \frac{A_{q\ul{z}+\ul{y}}C_{\ul{\psi},\ul{y}}\prod_{i=1}^{n}(qz_i)^{\psi_i}}{p^{s}(q\ul{z}+\ul{y})},\\
		\end{aligned}
	\end{equation}\endgroup
where 	$C_{\ul{\psi},\ul{y}}$ is uniquely defined by $\ul{y}$ for given $\ul{\psi}$. So the sequence $(C_{\ul{\psi}}(\ul{x}))_{\ul{x}>\ul{0}}$ defined by $$ C_{\ul{\psi}}(\ul{x})=C_{\ul{\psi},\ul{y}} \ \text{with} \ x_i \equiv y_i \mod q, 1 \leq i \leq n$$ is periodic as a function of $\ul{x}$, we let $Res_{\ul{\mu}}(s)$  denote the term $$Res_{\ul{\mu}}(s)=\sum_{(\ul{\psi})<(\ul{\mu})}\sum_{(\ul{y})<(\ul{q})}\sum_{(\ul{z}) \in \mathbf{N}^n/\left\{\ul{t}<\ul{N_0}\right\}} \frac{A_{q\ul{z}+\ul{y}}C_{\ul{\psi},\ul{y}}\prod_{i=1}^{n}(qz_i)^{\psi_i}}{p^{s}(q\ul{z}+\ul{y})}.$$ 
Remarking that all sequences $(b_{\ul{\psi}}({\ul{x}}))_{\ul{x}>\ul{0}}$ defined by $b_{\ul{\psi}}({\ul{x}})=A_{q\ul{z}+\ul{y}}C_{\ul{\psi},\ul{y}}$ if $\ul{x}=q\ul{z}+\ul{y}$ are in the form of a product of a specific $q$-automatic sequence by a $\ul{q}$-periodic one, because of Proposition 2.1, such sequences admit a unique completion the same one as $(A_{\ul{x}})_{\ul{x} \in \mathbf{N}_+^n}$  has, and the transition matrices on this completion do not depend on the choice of the $\ul{q}$-periodic sequences $(C_{\ul{\psi}}(\ul{x}))_{\ul{x}>\ul{0}}$.

Using the transition matrices, we have:
	\begingroup\small\begin{equation}
		\begin{aligned}
F_{\ul{\mu}}(s)&=\sum_{(\ul{x})\in \mathbf{N}^n/(\ul{0})}\frac{A_{\ul{x}}\prod_{i=1}^{n}x_i^{\mu_i}}{p(\ul{x})^{s}}=\sum_{(\ul{x})<(\ul{N_0q})} \frac{A_{\ul{x}}\prod_{i=1}^{n}x_i^{\mu_i}}{p(\ul{x})^{s}}+\sum_{(\ul{y})<(\ul{q})}\sum_{(\ul{z}) \in \mathbf{N}^n/\left\{\ul{t}<\ul{N_0}\right\}} \frac{A_{q\ul{z}+\ul{y}}\prod_{i=1}^{n}(qz_i+y_i)^{\mu_i}}{p^{s}(q\ul{z}+\ul{y})}\\
			&=\sum_{(\ul{x})<(\ul{N_0q})} \frac{A_{\ul{x}}\prod_{i=1}^{n}x_i^{\mu_i}}{p(\ul{x})^{s}}+\sum_{(\ul{y})<(\ul{q})}\sum_{(\ul{z}) \in \mathbf{N}^n/\left\{\ul{t}<\ul{N_0}\right\}} \frac{A_{q\ul{z}+\ul{y}}\prod_{i=1}^{n}(qz_i)^{\mu_i}}{p^{s}(q\ul{z}+\ul{y})}+Res_{\mu}(s)\\
			&=\sum_{(\ul{x})<(\ul{N_0q})} \frac{A_{\ul{x}}\prod_{i=1}^{n}x_i^{\mu_i}}{p(\ul{x})^{s}}+\sum_{(\ul{y})<(\ul{q})}M_{\ul{y}}\sum_{(\ul{z}) \in \mathbf{N}^n/\left\{\ul{t}<\ul{N_0}\right\}} \frac{A_{q\ul{z}}\prod_{i=1}^{n}(qz_i)^{\mu_i}}{p^{s}(q\ul{z})}\frac{1}{(1+\frac{p_{y}(\ul{z})}{p(\ul{z})})^s}+Res_{\mu}(s)\\
			&=\sum_{(\ul{x})<(\ul{N_0q})} \frac{A_{\ul{x}}\prod_{i=1}^{n}x_i^{\mu_i}}{p(\ul{x})^{s}}+\sum_{(\ul{y})<(\ul{q})}M_{\ul{y}}\sum_{(\ul{z}) \in \mathbf{N}^n/\left\{\ul{t}<\ul{N_0}\right\}} \frac{A_{\ul{z}}\prod_{i=1}^{n}(qz_i)^{\mu_i}}{p^{s}(q\ul{z})}\sum_{k \geq 0}\binom{s+k-1}{k}(\frac{-p_{y}(\ul{z})}{p(\ul{z})})^{k}+Res_{\mu}(s)\\
			&=\sum_{(\ul{x})<(\ul{N_0q})} \frac{A_{\ul{x}}\prod_{i=1}^{n}x_i^{\mu_i}}{p(\ul{x})^{s}}+q^{<\ul{\zeta},\ul{1}>-ns}\sum_{(\ul{y})<(\ul{q})}M_{\ul{y}}\sum_{k \geq 0}\binom{s+k-1}{k}\sum_{(\ul{z}) \in \mathbf{N}^n/\left\{\ul{t}<\ul{N_0}\right\}} \frac{A_{\ul{z}}\prod_{i=1}^{n}(z_i)^{\mu_i}(-p_{y}(\ul{z}))^k}{(p(\ul{z}))^{s+k}}\\
			&+Res_{\mu}(s).\\
		\end{aligned}
	\end{equation}\endgroup	
The above equation gives:
	\begingroup\small\begin{equation}
		\begin{aligned}
&(Id-q^{<\ul{\mu},\ul{1}>-ns}\sum_{(\ul{y})<(\ul{q})}M_{\ul{y}})F_{\ul{\mu}}(s)=q^{<\ul{\mu},\ul{1}>-ns}\sum_{(\ul{y})<(\ul{q})}M_{\ul{y}}\sum_{k \geq 1}\binom{s+k-1}{k}\sum_{(\ul{z}) \in \mathbf{N}^n/\left\{\ul{t}<\ul{N_0}\right\}} \frac{A_{\ul{z}}\prod_{i=1}^{n}(z_i)^{\mu_i}(-p_{y}(\ul{z}))^k}{(p(\ul{z}))^{s+k}}\\
&+\sum_{(\ul{x})<(\ul{N_0q})} \frac{A_{\ul{x}}\prod_{i=1}^{n}x_i^{\mu_i}}{p(\ul{x})^{s}}+Res_{\mu}(s).
		\end{aligned}
	\end{equation}\endgroup	
By multiplying by $com^t(Id-q^{<\ul{\mu},\ul{1}>-ns}\sum_{(\ul{y})<(\ul{q})}M_{\ul{y}})$ on both side, we have:
\begin{equation}
	\begin{aligned}
&det(Id-q^{<\ul{\mu},\ul{1}>-ns}\sum_{(\ul{y})<(\ul{q})}M_{\ul{y}})F_{\ul{\mu}}(s)=com^t(Id-q^{<\ul{\mu},\ul{1}>-ns}\sum_{(\ul{y})<(\ul{q})}M_{\ul{y}})(\sum_{(\ul{x})<(\ul{N_0q})} \frac{A_{\ul{x}}\prod_{i=1}^{n}x_i^{\mu_i}}{p(\ul{x})^{s}}+Res_{\mu}(s)\\
&+q^{<\ul{\mu},\ul{1}>-ns}\sum_{(\ul{y})<(\ul{q})}M_{\ul{y}}\sum_{k \geq 1}\binom{s+k-1}{k}\sum_{(\ul{z}) \in \mathbf{N}^n/\left\{\ul{t}<\ul{N_0}\right\}} \frac{A_{\ul{z}}\prod_{i=1}^{n}(z_i)^{\mu_i}(-p_{y}(\ul{z}))^k}{(p(\ul{z}))^{s+k}}).
	\end{aligned}
\end{equation}
		
Because of Proposition 2.2, the infinite sum $Res_{\mu}(s)$ converges absolutely when $\Re(s) >l+n-\frac{1}{2d}$, and the infinite sum $\sum_{(\ul{z}) \in \mathbf{N}^n/\left\{\ul{t}<\ul{N_0}\right\}} \frac{A_{\ul{z}}\prod_{i=1}^{n}(z_i)^{\mu_i}(-p_{y}(\ul{z}))^k}{(p(\ul{z}))^{s+k}}$ is also convergent and bounded when $\Re(s) >l+n-\frac{k}{2d}$. Equation 2.5 shows that all terms on the right-hand side present a meromorphic continuity for $\Re(s) > l+n-\frac{1}{2d}$, so that $F_{\ul{\mu}}(s)$ has a meromorphic continuation on the half-plane $\Re(s) > l+n-\frac{1}{2d}$. 

To guarantee that this argument works recursively over all rational numbers of type $l+n-\frac{k}{2d}$, we have to check that the meromorphic continuity of the $Res_{\mu}(s)$ can also be extended in this way. Using once more the above argument over $Res_{\mu}(s)$, we can deduce that this infinite sum can be extended as a meromorphic function on the half-plane $\Re(s) > l+n-\frac{2}{2d}$, however, with a new term  ``$Res_{\mu}(s)$'', let us call it $Res'_{\mu}(s)$. Once more, we have to do the same thing for $Res'_{\mu}(s)$.  But remarking that after each operation, the degree of the monomial at the numerator decreases strictly, so after finitely many times of such operation, the term $Res$ vanishes. This fact guarantees that the iteration can be done successively to prove the meromorphic continuation of $F_{\ul{\mu}}(s)$ on the whole complex plane.  

Furthermore, the poles of such a function can only be located at the zeros of the function $s \longrightarrow det(Id-q^{<\ul{\mu},\ul{1}>-ns}\sum_{(\ul{y})<(\ul{q})}M_{\ul{y}})$ for an arbitrary $\ul{\mu} \in \mathbf{N}^n_{+}$, so we conclude that all poles of function $F(s)$ are located in the set $$s=\frac{1}{n}(\frac{\log \lambda}{\log q} + \frac{2ik\pi}{\log q} -l),$$
with $\lambda$ any eigenvalue of the matrix $\sum_{(\ul{y})<(\ul{q})}M_{\ul{y}},k \in \mathbf{Z},l \in \mathbf{Z}$ and $\log$ is defined as complex logarithm.  

\end{proof}

\begin{proof}[Proof of Theorem 2.2]
	Let us write the polynomial $p(\ul{x})$ in the form $p(\ul{x})=p_d(\ul{x})+Res(\ul{x})$ where $p_d(\ul{x})$ is the homogeneous polynomial with maximum degree of $p(\ul{x})$, say $d$. By Lemma 1, $\frac{Res(\ul{x})}{p_d(\ul{x})} = \mathcal{O}(|x|^{-1/d})$. So for a given number $m \in \mathbf{R}$, there exist $x_0 \in \mathbf{N}_{+}$ and a positive integer $k_0$ such that for all $|\ul{x}| > x_0$ and all $k >k_0$, $|\frac{Res^k(\ul{x})\prod_{i=1}^{n}x_i^{\mu_i}}{p_d^{k+m}(\ul{x})}| <(\frac{1}{2^k})$. For any given half-plane $\left\{s|\mathcal{R}_e(s) > m,m \in \mathbf{R}\right\}$, take an integer $s_0 > \max\left\{k_0,|m|\right\}$, we can compute that  

	\begin{equation*}
		\begin{aligned}
		&|\sum_{|\ul{x}| > x_0} \frac{A_{\ul{x}}\prod_{i=1}^{n}x_i^{\mu_i}}{(p_d(\ul{x}))^s}\sum_{k=s_0+1}^{\infty}\binom{-s}{k}\frac{Res^k(\ul{x})}{p_d^k(\ul{x})}|=|\sum_{|\ul{x}| > x_0} \frac{A_{\ul{x}}}{(p_d(\ul{x}))^{s-m}}\sum_{k=s_0+1}^{\infty}\binom{-s}{k}\frac{Res^k(\ul{x})\prod_{i=1}^{n}x_i^{\mu_i}}{p_d^{k+m}(\ul{x})}|\\
&<\sum_{|\ul{x}| > x_0} \frac{A_{\ul{x}}}{(p_d(\ul{x}))^{\Re(s-m)}}|\sum_{k=s_0+1}^{\infty}\binom{-\Re(s)}{k}|\frac{Res^k(\ul{x})\prod_{i=1}^{n}x_i^{\mu_i}}{p_d^{k+m}(\ul{x})}\\
&<(\sum_{|\ul{x}| > x_0} \frac{A_{\ul{x}}}{(p_d(\ul{x}))^{\Re(s-m)}})\mathcal{O}_{s_0}(\frac{1+|s|^{s_0+1}}{2^{s_0}}).
		\end{aligned}
	\end{equation*}
The above fact shows that the function $\phi(s)=\sum_{|\ul{x}| > x_0} \frac{A_{\ul{x}}\prod_{i=1}^{n}x_i^{\mu_i}}{(p_0(\ul{x}))^s}\sum_{k=s_0+1}^{\infty}\binom{-s+k}{k}\frac{Res^k(\ul{x})}{p_0^k(\ul{x})}$ admits a holomorphic continuation over the half-plane $\left\{s|\Re(s) > m,m \in \mathbf{R}\right\}$. Now let us consider the equivalence as below:
	
	\begin{equation}
		\begin{aligned}
		\sum_{(\ul{x})\in \mathbf{N}^n_{+}/(\ul{0})} \frac{A_{\ul{x}}\prod_{i=1}^{n}x_i^{\mu_i}}{p(x)^{s}}&=\sum_{|\ul{x}| \leq x_0} \frac{A_{\ul{x}}\prod_{i=1}^{n}x_i^{\mu_i}}{p(x)^{s}}+\sum_{|\ul{x}| > x_0} \frac{A_{\ul{x}}\prod_{i=1}^{n}x_i^{\mu_i}}{(p_0(\ul{x}))^s}\frac{1}{(1+\frac{Res(\ul{x})}{p_0(\ul{x})})^s}\\
		&=\sum_{|\ul{x}| \leq x_0} \frac{A_{\ul{x}}\prod_{i=1}^{n}x_i^{\mu_i}}{p(x)^{s}}+\sum_{|\ul{x}| > x_0} \frac{A_{\ul{x}}\prod_{i=1}^{n}x_i^{\mu_i}}{(p_0(\ul{x}))^s}\sum_{k=0}^{s_0}\binom{-s+k}{k}\frac{Res^k(\ul{x})}{p_0^k(\ul{x})}+\phi(s).\\
		\end{aligned}
	\end{equation}
	For each $k \in \mathbf{N}$,
		\begin{equation}
		\begin{aligned}
		\sum_{|\ul{x}| > x_0} \frac{A_{\ul{x}}\prod_{i=1}^{n}x_i^{\mu_i}}{(p_0(\ul{x}))^s}\frac{Res^k(\ul{x})}{p_0^k(\ul{x})}=\sum_{\ul{j} \leq k\ul{\mu}}\sum_{|\ul{x}| > x_0} \frac{A_{\ul{x}}\mathcal{C}_{\ul{i}}}{(p_0(\ul{x}))^{s+k}}\prod_{i=1}^{n}x_i^{j_i},\\
		\end{aligned}
		\end{equation}
where $\mathcal{C}_{\ul{i}}$ are constants depending on $k$, and $s \longrightarrow \sum_{|\ul{x}| > x_0} \frac{A_{\ul{x}}\mathcal{C}_{\ul{i}}}{(p_0(\ul{x}))^{s+k}}\prod_{i=1}^{n}x_i^{j_i}$ are meromorphic functions because of the previous lemma. As there are finitely many meromorphic function in (2.8), we can conclude that for every $k>0$, $s \longrightarrow \sum_{|\ul{x}| > x_0} \frac{A_{\ul{x}}\prod_{i=1}^{n}x_i^{\mu_i}}{(p_0(\ul{x}))^s}Res^k(\ul{x})$ is meromorphic. This fact implies that for an arbitrary $s_0 \in \mathbf{R}$ the function $f(s)$ is a finite sum of meromorphic functions on the half-plane $\Re(s)>s_0$, so $f(s)$ itself is meromorphic on this half plane. As a result, the function $s \longrightarrow \sum_{(\ul{x})\in \mathbf{N}^n/(\ul{0})} \frac{a_{\ul{x}}\prod_{i=1}^{n}x_i^{\mu_i}}{p(\ul{x})^{s}}$ is meromorphic on the whole complex plane.

\end{proof}

\begin{proposition}
	Let $f(s)=\sum_{(\ul{x})\in \mathbf{N}^n/(\ul{0})} \frac{a_{\ul{x}}\prod_{i=1}^{n}x_i^{\mu_i}}{p(\ul{x})^{s}}$ be the function defined as in Theorem 2.2. Let $s_0$ be its first pole on the axis of real numbers counting from plus infinity to minus infinity. Then the function $H(s)$ has a simple pole at this point.
\end{proposition}
	
\begin{proof}
 We recall a classical result on matrices (see \cite{minc}): Let $B$ be a matrix of size $t \times t$ over any commutative field, $p_B(X)$ be its characteristic polynomial, and $\pi_B(X)$ be its monic minimal polynomial. Let $\Delta(X)$ be the monic gcd of the entries of (the
transpose of) the comatrix of the matrix $(B-XI)$, then:
$$p_B(X)=(-1)^t\pi_B(X)\Delta(X)$$
We let $B$ denote the matrix $(nq)^{-1}\sum_{(\ul{y})<(\ul{q})}M_{\ul{y}}$ and by $T$ its size. By dividing by $\Delta(q^{n(s-1)+<\ul{\mu},\ul{1}>})$ both sides of Formula 2.6, we get:
 \begin{equation}
	\begin{aligned}
&\pi_B(q^{n(s-1)+<\ul{\mu},\ul{1}>})H(s)=\frac{com^t(q^{n(s-1)+<\ul{\mu},\ul{1}>}Id-\sum_{(\ul{y})<(\ul{q})}M_{\ul{y}})}{\Delta(q^{n(s-1)+<\ul{\mu},\ul{1}>})}
(\sum_{(\ul{x})<(\ul{N_0q})} \frac{A_{\ul{x}}\prod_{i=1}^{n}x_i^{\mu_i}}{p(\ul{x})^{s}}+Res_{\mu}(s)\\
&+q^{<\ul{\mu},\ul{1}>-ns}\sum_{(\ul{y})<(\ul{q})}M_{\ul{y}}\sum_{k \geq 1}\binom{-s+k}{k}\sum_{(\ul{z}) \in \mathbf{N}^n/\left\{\ul{t}<\ul{N_0}\right\}} \frac{A_{\ul{z}}\prod_{i=1}^{n}(z_i)^{\mu_i}(-p_{y}(\ul{z}))^k}{(p(\ul{z}))^{s+k}}).
	\end{aligned}
\end{equation}
The right-hand side of the above function is holomorphic when $\Re(s) > s_0$. As $s_0$ is the first pole of $H(s)$ on the real axis counting from plus infinity, it is a zero of the function $\pi_B(q^{n(s-1)+<\ul{\mu},\ul{1}>})$ associated with the eigenvalue 1 of the matrix $B$. On the other hand, as $B$ is a stochastic matrix, $\pi_B(x)$ has a simple root at 1, so the function $\pi_B(q^{n(s-1)+<\ul{\mu},\ul{1}>})$ has a simple root at $s_0$ which concludes the proposition.

\end{proof}

\section{Review and remarks}

The critical point of the above proof is the development of the term $(1-\frac{p_d(\ul{x}+\ul{\mu})-p_d(\ul{x})}{p_d(\ul{x})})^{-s}$ as an infinite sum, which works only if $|\frac{p_d(\ul{x}+\ul{\mu})-p_d(\ul{x})}{p_d(\ul{x})}| <1$. As we have to use this fact successively to deal with the term $Res(s)$ in Proposition 2.3, what we need actually is that $|\frac{\partial_{\ul{\mu}}(p_d)}{p_d}|$ is bounded by $1$ for all $\ul{\mu}$ such that $<\ul{\mu},\ul{1}> \leq d$ when $||\ul{x}||$ is large. The assumption of ellipticity of the polynomial is a particular case of the above propriety. So we may expect to achieve the same result under the assumption $$|\frac{\partial_{\ul{\mu}}(p_d)}{p_d}(\ul{x})| <1, \text{for all}\; <\ul{\mu},\ul{1}> \leq d,x_i \geq 1.$$
We may compare this assumption with that in \cite{ess}, saying
$$\frac{\partial_{\ul{\mu}}(p_d)}{p_d}(\ul{x}) =\mathcal{O}(1),x_i \geq 1,$$
which is the weakest assumption known to have a meromorphic continuation of Dirichlet series $\sum_{\ul{x}\in \mathbf{N}^n_+} \frac{1}{P(\ul{x})^s}$.

\section{Infinite products}

Let $P(x)=\sum_{i=0}^{d} a_ix^i$ be a polynomial which does not have zeros on $\mathbb{Q}$ and $\tilde{P}(x)$ be the polynomial defined by $\tilde{P}(x)=\sum_{i=0}^{d-1} -\frac{a_i}{a_d}x^{n-i}$, by definition,  we have $P(x)=a_dx^d-a_dx^d\tilde{P}(\frac{1}{x})$. Let us define $c_i=\frac{a_i}{a_d}$ for all $i=0,1,...,d-1$.

In this section we consider two Dirichlet series generated by $1$-index automatic sequences:$$f(s)=\sum_{n=0}^{\infty} \frac{(\zeta)^{S_q(n)}}{(P(n+1))^s}$$
$$g(s)=\sum_{n=1}^{\infty} \frac{(\zeta)^{S_q(n)}}{(P(n))^s},$$
where $q$ and $r$ are two integers satisfying $2\leq r \leq q$ and $r$ divides $q$, $\zeta$ is a $r$-th root of unity, such that $\zeta \neq 1$, $S_q(n)$  is the sum of digits of $n$ in the $q$-ary expansion satisfying $S_q(0)=0$ and $S_q(qn+a)=S_q(n)+a$ for $0\leq a\leq q-1$.\\

Let us define:
$$\phi(s)=\sum_{n=0}^{\infty} \frac{(\zeta)^{S_q(n)}}{(n+1)^s}$$
$$\psi(s)=\sum_{n=1}^{\infty} \frac{(\zeta)^{S_q(n)}}{(n)^s}.$$
It is proved in \cite{jpc} that $\phi$ and $\psi$ have holomorphic continuations to the whole complex plane, and $\psi(s)(q^s-1)=\phi(s)(\zeta q^s-1)$ for all $s \in \mathbb{C}$. \\

\begin{proposition}

$f$ and $g$ also have holomorphic continuations to the whole complex plane if $c_1, c_2,...$ satisfy $max{|c_i|} < \frac{1}{d}$.

\end{proposition}

\begin{proof}

We firstly remark that the hypothesis of $\max{|c_i|} < \frac{1}{d}$ implies the fact $|\tilde{P}(\frac{1}{n+1})|<1$ for any $n \in \mathbf{N}_{+}$. Indeed, $|\tilde{P}(\frac{1}{n+1})| = |\sum_{i=0}^{d-1} -\frac{a_i}{a_d}(n+1)^{i-d}| \leq \sum_{i=0}^{d-1} |\frac{a_i}{a_d}|<1$.
\begin{equation}
          \begin{aligned}
          f(s)&=\sum_{n=0}^{\infty} \frac{(\zeta)^{S_q(n)}}{(P(n+1))^s}\\
		&=a_d^{-s}\sum_{n=0}^{\infty} \frac{(\zeta)^{S_q(n)}}{(n+1)^{ds}(1-\tilde{P}(\frac{1}{n+1}))^s}\\
		&=a_d^{-s}\sum_{n=0}^{\infty} \frac{(\zeta)^{S_q(n)}}{(n+1)^{ds}}\sum_{k=0}^{\infty}\binom{s+k-1}{k}\tilde{P}(\frac{1}{n+1}))^k\\
		&=a_d^{-s}\sum_{n=0}^{\infty} \frac{(\zeta)^{S_q(n)}}{(n+1)^{ds}}\sum_{k=0}^{\infty}\binom{s+k-1}{k}\sum_{l=k}^{dk}m_{k,l}(n+1)^{-l}\\
	        &=a_d^{-s}\sum_{k=0}^{\infty}\binom{s+k-1}{k}\sum_{l=k}^{dk}m_{k,l}\sum_{n=0}^{\infty} \frac{(\zeta)^{S_q(n)}}{(n+1)^{ds+l}}\\
           &=a_d^{-s}\sum_{k=0}^{\infty}\binom{s+k-1}{k}\sum_{l=k}^{dk}m_{k,l}\phi(ds+l)\\
	&=a_d^{-s}\phi(ds)+a_d^s\sum_{k=1}^{\infty}\binom{s+k-1}{k}\sum_{l=k}^{dk}m_{k,l}\phi(ds+l),
          \end{aligned}
\end{equation}
where $m_{k,l}=\sum_{M_{k,l} \in \mathcal{P}(\left\{c_i| 1\leq i\leq n-1\right\})}\prod_{c_i \in M_{k,l}}c_i$ and $M_{k,l}$ are sets of $k$ elements included in $\left\{c_i| 1\leq i\leq n-1\right\}$ and the sum of indices of its elements equals $l$. The hypothesis $m=\max{|ci|} < \frac{1}{d}$ shows that $|\sum_{l=k}^{dk-k}m_{k,l}| \leq (md)^k<1$, so the right-hand side of (2.10) converges uniformly over the half plane $\left\{s|\Re(s) > 0\right\}$ because $\phi(s)$ is bounded for large $|s|$. By the same argument as in Theorem 2.2, we prove successively the holomorphic continuation of $f$ on the whole complex plane. 
\end{proof}
It is easy to check $f(0)=0$, and because of the uniform convergence of the right-hand side of (2.10) over the half-plane $\left\{s|\Re(s) > 0\right\}$, dividing  by $s$ and  letting $s$ tend  to $0$ we have:

\begin{equation}
          \begin{aligned}
          f'(0)&= d\phi'(0)+a_d^{-s}\sum_{k=1}^{\infty}\lim_{s\to 0}\frac{1}{s}\binom{s+k-1}{k}\sum_{l=k}^{dk}m_{k,l}\phi(ds+l)\\
		&=-d\log q/(\zeta -1)+\sum_{k=1}^{\infty}k^{-1}\sum_{l=k}^{dk}m_{k,l}\phi(l)\\
		&=-d\log q/(\zeta -1)+\sum_{k=1}^{\infty}k^{-1}\sum_{l=k}^{dk}m_{k,l}\sum_{n=0}^{\infty} \frac{(\zeta)^{S_q(n)}}{(n+1)^l}\\
		&=-d\log q/(\zeta -1)+\sum_{n=0}^{\infty}(\zeta)^{S_q(n)}\sum_{k=1}^{\infty}k^{-1}\sum_{l=k}^{dk}m_{k,l} \frac{1}{(n+1)^l}\\
		&=-d\log q/(\zeta -1)+\sum_{n=0}^{\infty}(\zeta)^{S_q(n)}\sum_{k=1}^{\infty}k^{-1}\tilde{P}^k(\frac{1}{n+1})\\
	        &=-d\log q/(\zeta -1)+\sum_{n=0}^{\infty}(\zeta)^{S_q(n)}\log(1-\tilde{P}(\frac{1}{n+1})),
          \end{aligned}
\end{equation}on the other hand, one has for all $s$, $\psi(s)(q^s-1)=\phi(s)(\zeta q^s-1).$

\begin{equation}
          \begin{aligned}
          f'(0)&=d\phi'(0)+a_d^{-s}\sum_{k=1}^{\infty}\lim_{s\to 0}\frac{1}{s}\binom{s+k-1}{k}\sum_{l=k}^{dk}m_{k,l}\phi(ds+l)\\
		&=-d\log q/(\zeta -1)+\sum_{k=1}^{\infty}k^{-1}\sum_{l=k}^{dk}m_{k,l}\psi(l)(q^l-1)/(\zeta q^l-1)\\
		&=-d\log q/(\zeta -1)+\zeta^{-1}\sum_{k=1}^{\infty}k^{-1}\sum_{l=k}^{dk}m_{k,l}\psi(l)+(\zeta^{-1}-1)\sum_{k=1}^{\infty}k^{-1}\sum_{l=k}^{dk}m_{k,l}\psi(l)/(\zeta q^l-1).
          \end{aligned}
\end{equation}
By the same method as above, we can deduce by calculating $g'(0)$ $$\sum_{k=1}^{\infty}k^{-1}\sum_{l=k}^{dk}m_{k,l}\psi(l)=\sum_{n=1}^{\infty}(\zeta)^{S_q(n)}\log(1-\tilde{P}(\frac{1}{n}))$$
and

\begin{equation}
          \begin{aligned}
          \sum_{k=1}^{\infty}k^{-1}\sum_{l=k}^{dk}m_{k,l}\psi(l)/(\zeta q^l-1)&= \sum_{k=1}^{\infty}k^{-1}\sum_{l=k}^{dk}m_{k,l}\sum_{n=1}^{\infty}\frac{(\zeta)^{S_q(n)}}{(n)^l}\sum_{r=1}^{\infty}(\zeta q^l)^{-r}\\
		&=\sum_{r=1}^{\infty}(\zeta)^{-r}\sum_{n=1}^{\infty}(\zeta)^{S_q(n)}\sum_{k=1}^{\infty}k^{-1}\sum_{l=k}^{dk}m_{k,l}\frac{1}{(nq^r)^l}\\
		&=\sum_{r=1}^{\infty}(\zeta)^{-r}\sum_{n=1}^{\infty}(\zeta)^{S_q(n)}\log(1-\tilde{P}(\frac{1}{nq^r}))\nonumber.
          \end{aligned}
\end{equation}
As a consequence, 

\begin{equation}\sum_{n=0}^{\infty}(\zeta)^{S_q(n)}\log(1-\tilde{P}(\frac{1}{n+1}))=\zeta^{-1}\sum_{n=1}^{\infty}(\zeta)^{S_q(n)}\log(1-\tilde{P}(\frac{1}{n}))+(\zeta^{-1}-1)\sum_{r=1}^{\infty}\sum_{n=1}^{\infty}(\zeta)^{S_q(n)-r}\log(1-\tilde{P}(\frac{1}{nq^r})).\end{equation}

\begin{proposition}
 We have the equality
$$\prod_{n=0}^{\infty}(1-\tilde{P}(\frac{1}{n+1}))^{\zeta^{Sq_n}}\times\prod_{n=1}^{\infty}(1-\tilde{P}(\frac{1}{n}))^{-\zeta^{Sq_n-1}}\times(\prod_{r=1}^{\infty}\prod_{n=1}^{\infty}(1-\tilde{P}(\frac{1}{nq^r}))^{\zeta^{Sq_n}-r})^{1-\zeta^{-1}}=1$$
\end{proposition}

Moreover, if we suppose, for any $j$
$$
x_j(m) =
\begin{cases}
\frac{r-1}{r} & \text{if $s_q(m)=j$ (mod r)}\\
-\frac{1}{r} & \text {if $s_q(m)\neq j$ (mod r)}\\
\end{cases}
$$
We clearly have $$\sum_{j \mod r} x_j(m)=0.(*)$$
Furthermore, $$\sum_{j \mod r} x_j(m)\zeta^j=\zeta^{S_q(m)}.(**)$$
Using $(*)$, Formula (2.13) can be reformulated as
\begin{equation}
          \begin{aligned}
&\sum_{j\mod r}\zeta^j\sum_{n=0}^{\infty}x_j(n)(\log(1-\tilde{P}(\frac{1}{n+1})+\sum_{r=1}^{\infty}\zeta^{-r}\log(1-\tilde{P}(\frac{1}{nq^r})))=\\
&\sum_{j\mod r}\zeta^{j-1}\sum_{n=1}^{\infty}x_j(n)(\log(1-\tilde{P}(\frac{1}{n})+\sum_{r=1}^{\infty}\zeta^{-r}\log(1-\tilde{P}(\frac{1}{nq^r}))).
          \end{aligned}
\end{equation}
Let now $\eta$ be a primitive root of unity, we can apply relation (2.14) successively to $\zeta=\eta^a$ for $a=1,2,...,r-1$. \\
Because of $(**)$, we also have 

\begin{equation}
          \begin{aligned}
&\sum_{j\mod r}\sum_{n=0}^{\infty}x_j(n)(\log(1-\tilde{P}(\frac{1}{n+1}))+\sum_{r=1}^{\infty}\zeta^{-r}\log(1-\tilde{P}(\frac{1}{nq^r})))=\\
&\sum_{j\mod r}\sum_{n=1}^{\infty}x_j(n)(\log(1-\tilde{P}(\frac{1}{n}))
+\sum_{r=1}^{\infty}\zeta^{-r}\log(1-\tilde{P}(\frac{1}{nq^r})))=0.
          \end{aligned}
\end{equation}
Define the matrices: $Mat_1$ to be $Mat_1=(\eta^{ij})$ and $Mat_2$ to be $Mat_2=(\eta^{ij-i})$, $i=0,1,...,r-1; j=0,1,...,r-1$, define $\lambda$ and $\beta$ by
$$\lambda(j)=\sum_{n=0}^{\infty}x_j(n)(\log(1-\tilde{P}(\frac{1}{n+1}))+\sum_{r=1}^{\infty}\zeta^{-r}\log(1-\tilde{P}(\frac{1}{nq^r})))$$
$$\beta(j)=\sum_{n=1}^{\infty}x_j(n)(\log(1-\tilde{P}(\frac{1}{n}))+\sum_{r=1}^{\infty}\zeta^{-r}\log(1-\tilde{P}(\frac{1}{nq^r}))).$$
Let 

\[ A=\left( \begin{array}{c}
\lambda(0)\\
\lambda(1)\\
...\\
\lambda(r-1) \end{array} \right)\] 

\[ B=\left( \begin{array}{c}
\beta(0) \\
\beta(1) \\
...\\
\beta(r-1) \end{array}\right).\] 

Then we have $$Mat_1A=Mat_2B.$$

On the other hand, $A$ is invertible and $Mat_1=Mat_2\times Mat_3$ with

\[ Mat_3= \left( \begin{array}{ccccc}
0 & 0 & ...& 0 & 1 \\
1 & 0 & ...& 0 & 0\\
0 & 1 & ... & 0 & 0\\
0 & 0 & ... & 1 & 0 \end{array} \right).\]

So we have $$A=Mat_3 \times B.$$

\begin{proposition}
 We have the equality $\lambda(i)=\beta(i-1)$ for $i=1,2,...,r-1$ and  $\lambda(0)=\beta(r-1)$, which leads to,  for $i=1,2,...,r-1$,
$$\prod_{n=0}^{\infty}((1-\tilde{P}(\frac{1}{n+1})\times\prod_{r=1}^{\infty}(1-\tilde{P}(\frac{1}{nq^r})))^{x_j(n)}=\prod_{n=1}^{\infty}((1-\tilde{P}(\frac{1}{n})\times\prod_{r=1}^{\infty}(1-\tilde{P}(\frac{1}{nq^r})))^{x_{j-1}(n)},$$
and for $i=0$,
$$\prod_{n=0}^{\infty}((1-\tilde{P}(\frac{1}{n+1})\times\prod_{r=1}^{\infty}(1-\tilde{P}(\frac{1}{nq^r})))^{x_0(n)}=\prod_{n=1}^{\infty}((1-\tilde{P}(\frac{1}{n})\times\prod_{r=1}^{\infty}(1-\tilde{P}(\frac{1}{nq^r})))^{x_{r-1}(n)}.$$
\end{proposition}

\section{Acknowledgement}
We would like to thank D. Essouabri for his helpful remarks and discussions on the meromorphic continuation of general Dirichlet series.

\bibliographystyle{alpha}
\bibliography{citations}

\begin{thebibliography}{AFP00}

\bibitem[AFP00]{ALLOUCHE2000}
J.-P. Allouche, M.~Mendès France, and J.~Peyrière.
\newblock {Automatic Dirichlet series}.
\newblock {\em Journal of Number Theory}, 81(2):359 -- 373, 2000.

\bibitem[CA85]{jpc}
H.~Cohen and J.~P. Allouche.
\newblock {Dirichlet Series and Curious infinite Products}.
\newblock {\em Bulletin of the London Mathematical Society}, 17(6):531--538, 11
  1985.

\bibitem[Ess97]{ess}
D.~Essouabri.
\newblock {Singularit\'e de s\'eries de Dirichlet associ\'ees \`a des
  polyn\^omes de plusieurs variables et applications en th\'eorie analytique
  des nombres}.
\newblock {\em Annales de l'Institut Fourier}, 47(2):429--483, 1997.

\bibitem[Mah28]{Mah}
K.~Mahler.
\newblock Über einen satz von mellin.
\newblock {\em Mathematische Annalen}, 100:384--398, 1928.

\bibitem[Mel00]{Mellin}
R.~H. Mellin.
\newblock {\em Acta Soc. Scient. Fennicae}, 4:29, 1900.

\bibitem[Min88]{minc}
H.~Minc.
\newblock {\em Nonnegative matrices}.
\newblock Wiley, 1988.

\bibitem[Sal86]{salon}
O.~Salon.
\newblock Suites automatiques a multi-indices.
\newblock {\em Séminaire de Théorie des Nombres de Bordeaux}, pages 1--36A,
  1986.

\bibitem[Sar84]{sargos}
P.~Sargos.
\newblock Prolongement m\'eromorphe des s\'eries de dirichlet associ\'ees \`a
  des fractions rationnelles de plusieurs variables.
\newblock {\em Annales de l'Institut Fourier}, 34(3):83--123, 1984.

\end{thebibliography}

\end{document}